\documentclass[12pt]{amsart}
\usepackage{amssymb, amsmath, mathabx, fdsymbol}
\newtheorem{theorem}{Theorem}

\newtheorem{lemma}{Lemma}
\newtheorem{corollary}{Corollary}
\newtheorem{question}{Question}

\title{Dense metrizable subspaces in powers of Corson compacta}

\author{Arkady Leiderman}

\address{Department of Mathematics\\
Ben Gurion University of the Negev\\
P.O.B. 653\\
Be'er Sheva 8410501\\
Israel}

\email{arkady@math.bgu.ac.il}

\author{Santi Spadaro}

\address{Dipartimento di Ingegneria\\
Universit\`a degli Studi di Palermo\\
viale delle Scienze 8\\
90128 Palermo, Itay}

\email{santidomenico.spadaro@unipa.it}

\author{Stevo Todorcevic}

\address{Department of Mathematics \\
University of Toronto\\
Toronto, ON, Canada, M5S 2E4}

\email{stevo@math.toronto.edu}

\address{Institut de Math\'ematiques de Jussieu\\
Paris, France}
\email{stevo.todorcevic@cnrs.fr}

\address{Mathematical Institute, SASA\\
Belgrade, Serbia}
\email{stevo.todorcevic@sanu.ac.rs}



\begin{document}
\maketitle

\begin{abstract}
We characterize when the countable power of a Corson compactum has a dense metrizable subspace and construct consistent examples of Corson compacta whose countable power does not have a dense metrizable subspace. We also give several remarks about ccc Corson compacta and, as a byproduct, we obtain a new proof of Kunen and van Mill's characterization of when a Corson compactum supporting a strictly positive measure is metrizable.
\end{abstract}

\section{Introduction}

A compact space $X$ is called a \emph{Corson Compactum} if it can be embedded into a $\Sigma$-product of lines, that is, if there is a cardinal $\kappa$, a space $Y \subset \Sigma(\mathbb{R}^\kappa):=\{x \in \mathbb{R}^\kappa: |supp{(x)}| \leq \aleph_0 \}$ and a homeomorphism $f: X \to Y$.

Corson compacta are a class of spaces which is relevant to both topology and functional analysis. For example, every Eberlein compactum (i.e., a weakly compact subspace of a Banach space) is a Corson compactum and, for every Banach space $X$, if the unit ball in the dual $X^*$ is a Corson compactum then $X$ is Lindel\"of in the weak topology.

Shapirovskii \cite{Sh} proved that every Corson compactum has a dense first-countable subspace. Eberlein compacta even have dense metrizable subspaces (see \cite{AL}, Theorem 5), and this prompted the question of whether that was also true for Corson compacta. Gruenhage and the first author proved independently that the answer to this question is positive for Gul'ko compacta, a class of Corson compacta wider than that of Eberlein compacta (see \cite{Le} and \cite{Gr}). It is relatively easy to come up with a consistent example of a Corson compact space without a dense metrizable subspace (see the final section of our paper), but the first ZFC example of a Corson compact space without a dense metrizable subspace was only constructed by the third author in 1981 (see \cite{TCors}, \cite{THand}).

Clearly, the countable power of a Corson compactum is still a Corson compactum. This suggested us to investigate whether such a space must have a dense metrizable subspace. We first present a characterization of when the countable power of a Corson compactum has a dense metrizable subspace. Namely, we prove that if $X$ is a Corson compactum, then $X^\omega$ has a dense metrizable subspace if and only if it has a \emph{large} family of pairwise disjoint non-empty open sets. Exploiting this, we then show a few consistent examples of Corson compact spaces whose countable power does not contain a dense metrizable subspace. It is an open question whether such a space exists in ZFC. We also prove that the existence of a ccc counterexample to our question is equivalent to the failure of $MA_{\omega_1}$ for powerfully ccc posets and give a new proof to Kunen and van Mill's theorem stating that the existence of a non-metrizable Corson compactum with a strictly positive measure is equivalent to the failure of $MA_{\omega_1}$ for measure algebras.

All topological spaces are assumed to be at least Hausdorff. All measures are assumed to be Radon probability measures. For undefined notions and notation we refer the reader to \cite{En}, \cite{Je} and \cite{Ju}.

\section{Powers of Corson compact spaces}

Recall that a $\pi$-base $\mathcal{P}$ for a topological space $X$ is a family of non-empty open subsets of $X$ such that for every non-empty open set $U \subset X$, there is $P \in \mathcal{P}$ such that $P \subset U$. The $\pi$-weight of $X$ ($\pi w(X)$) is defined as the minimum cardinality of a $\pi$-base for $X$. Clearly, if $d(X)$ denotes the density of $X$, we have $d(X) \leq \pi w(X)$. For Corson compacta it is known that $\pi w(X)=d(X)=w(X)$ (see \cite{M}).

The following lemma is well-known. It was noted, for example, in \cite{Le}.

\begin{lemma} \label{LemArk}
A Corson compactum has a dense metrizable subspace if and only if it has a $\sigma$-disjoint $\pi$-base.
\end{lemma}

\begin{proof}
It suffices to combine Shapirovskii's result from \cite{Sh} that every Corson compactum has a dense subspace of points of countable character with White's result from \cite{Wh} that for first-countable spaces, having a dense metrizable subspace is equivalent to having a $\sigma$-disjoint $\pi$-base.
\end{proof}

Recall that a \emph{cellular family} in a topological space $X$ is a family of pairwise disjoint non-empty open subsets of $X$.  The \emph{cellularity of $X$} ($c(X)$) is defined as the supremum of the cardinalities of cellular families in $X$. The cardinal function $\hat{c}(X)$ is defined as the minimum cardinal $\kappa$ such that $X$ does not have a cellular family of cardinality $\kappa$. Obviously $c(X) \leq \hat{c}(X)$ and $c(X) \leq \pi w(X)$.

The cardinal function $\hat{c}(X)$, in conjunction with the $\pi$-weight, can be used to characterize when the countable power of a Corson compact space has a dense metrizable subspace. We first need the following technical lemma.

\begin{lemma} \label{LemMain}
Let $X$ be any space and $Y$ be a space such that $\hat{c}(Y^\omega) > \pi w(Y) \geq \pi w(X)$. Then $X \times Y^\mu$ has a $\sigma$-disjoint $\pi$-base for every cardinal $\mu \in [\omega , \pi w(X)]$.
\end{lemma}

\begin{proof}
Let $\{A_k: k < \omega \}$ be an increasing sequence of subsets of $\mu$ such that $\mu \setminus A_k$ is infinite, for every $k<\omega$ and $\bigcup \{A_k: k < \omega \} = \mu$. Let $\kappa=\pi w(Y)$.

Let $\mathcal{B}_X$ be a $\pi$-base for $X$ and let $\mathcal{B}_Y$ be a $\pi$-base for $Y$ such that $|\mathcal{B}_X| \leq |\mathcal{B}_Y| =\kappa$.

For every $a \in [\mu]^{<\omega}$ fix an enumeration $\{B^a_\xi: \xi < \kappa \}$ of the set $\mathcal{B}_X \times \mathcal{B}_Y^a$. Use the fact that $\mu \setminus A_k$ is infinite and $\hat{c}(Y^\omega) > \kappa$ to fix a cellular family $\{C^k_\xi: \xi < \kappa \}$ of size $\kappa$ in $Y^{\mu \setminus A_k}$.

Let $\{a^n_\alpha: \alpha < \mu \}$ be an enumeration of $[A_n]^{<\omega}$. Moreover, fix a partition $\{B_\alpha: \alpha < \mu \}$ of $\kappa$ into $\mu$ many sets of size $\kappa$ and enumerate $B_\alpha$ as $\{b^\alpha_\xi: \xi <\kappa \}$.

Define $\pi_{-1}: X \times Y^\mu \to X$ to be the usual projection onto $X$ and let:

$U(n, \alpha, \xi)=U(n,\alpha, \xi)_{-1} \times \prod_{\gamma < \mu} U(n, \alpha, \xi)_\gamma$, where:

$$
U(n, \alpha, \xi)_\gamma=\begin{cases}
\pi_{-1}(B^{a^n_\alpha}_\xi) & \mbox{if } \gamma=-1 \\
\pi_\gamma(B^{a^n_\alpha}_\xi) & \mbox{if } \gamma \in a^n_\alpha \\
Y & \mbox{if } \gamma \in A_n \setminus a^n_\alpha \\
\pi_\gamma(C^n_{b^\alpha_\xi}) & \mbox{if } \gamma \in \mu \setminus A_n
\end{cases}
$$

Note that $\mathcal{U}_n=\{U(n, \alpha, \xi): \alpha < \mu, \xi<\kappa \}$ is a disjoint collection of non-empty open sets and $\mathcal{U}=\bigcup_{n<\omega} \mathcal{U}_n$ is a $\pi$-base for $X \times Y^\mu$. To see the latter, let $W$ be a basic open set in $X \times Y^\mu$ and $s=supp(W)$.

Let $n<\omega$ be large enough so that $s \subset A_n$. Choose $\alpha < \mu$ so that $s \cap \mu=a^n_\alpha$ and $\xi < \kappa$ so that $B^{a^n_\alpha}_\xi \subset \tilde{W}$, where $\tilde{W}=\pi_{-1}(W) \times \prod_{\gamma \in a^n_\alpha} \pi_\gamma(W)$. Then we also have that $U(n, \alpha, \xi) \subset W$ and hence we are done.
\end{proof}

A first consequence of the above Lemma is that given any Corson compactum, we can find another Corson compactum whose product with it contains a dense metrizable subspace.

\begin{corollary}
Let $X$ be any Corson compactum. Then there is a Corson compactum $K$ such that $X \times K$ has a dense metrizable subspace.
\end{corollary}

\begin{proof}
Let $\kappa=\pi w(X)$ and denote with $A(\kappa)$ the one-point compactification of a discrete set of cardinality $\kappa$. Let $K=A(\kappa)^\omega$. Then $K$ is a Corson compactum and by Lemma $\ref{LemMain}$, $X \times K$ has a $\sigma$-disjoint $\pi$-base and hence a dense metrizable subspace.
\end{proof}

\begin{theorem} \label{mainthm}
Let $X$ be a Corson compactum. Then $X^\omega$ has a dense metrizable subspace if and only if $\hat{c}(X^\omega) > w(X)$.
\end{theorem}

\begin{proof}
First of all recall that $d(X)=\pi w(X)=w(X)$, for every Corson compactum $X$.

The reverse implication follows by combining Lemma $\ref{LemMain}$ with Lemma $\ref{LemArk}$. For the direct implication, let $X$ be a Corson compactum such that $X^\omega$ has a dense metrizable subspace. Then $X^\omega$ has a $\sigma$-disjoint $\pi$-base $\mathcal{B}=\bigcup_{n<\omega} \mathcal{B}_n$, where $\mathcal{B}_n$ is pairwise disjoint. Since $\pi w(X^\omega)=\pi w(X)$ we must have $|\mathcal{B}| \geq \pi w(X)$. We consider two cases.

\begin{enumerate}

\item \noindent {\bf $\pi w(X)$ has uncountable cofinality:} then there must be $n<\omega$ such that $|\mathcal{B}_n| \geq \pi w(X)$. So $X^\omega$ has a cellular family of cardinality $\geq \pi w(X)$. 

\item \noindent {\bf $\pi w(X)$ has countable cofinality:} then let $\{\kappa_n: n < \omega\} \subset \pi w(X)$ be a sequence of cardinals such that $\sup \{\kappa_n: n < \omega \}=\pi w(X)$. Then for each $n<\omega$ there must be $k < \omega$ such that $|\mathcal{B}_k| \geq \kappa_n$. Therefore $c(X^\omega) = \pi w(X)$. Since $\pi w(X)$ is a singular cardinal, by 4.1 of \cite{Ju} we have $\hat{c}(X^\omega)=\pi w(X)^+$, and we are done.
\end{enumerate}
 \end{proof}
 
 \begin{question} \label{mainquest}
 Is there a Corson compactum $X$ such that $X^\omega$ has no dense metrizable subspace?
 \end{question}
 
In the next few sections we will provide various consistent examples answering the above question in the positive.

\section{Corson compacta with the countable chain condition}

One way to construct an example of a Corson compactum whose countable power doesn't have a dense metrizable subspace would be to construct a Corson compactum $X$ such that $X^\omega$ is ccc but $X$ is not separable (or equivalently, $X$ is not metrizable). There is no way to do this in ZFC since, according to a result of Shapirovskii (see, for example, Theorem 3.1 and Corollary 3.2 of \cite{TChain}) under $MA_{\omega_1}$ every compact countably tight ccc space is separable. This is actually equivalent to $MA_{\omega_1}$, as the following theorem shows.

\begin{theorem} \label{cccequiv}
The following are equivalent:

\begin{enumerate}
\item \label{cccCorson} There is a ccc Corson compactum without a dense metrizable subspace.
\item \label{nwdCompact} There is a compact ccc space which can be covered by $\omega_1$ nowhere dense subsets.
\item \label{notMAomega1} $MA_{\omega_1}$ fails
\end{enumerate}
\end{theorem}

\begin{proof}
Items ($\ref{nwdCompact}$) and ($\ref{notMAomega1}$) are well known to be equivalent (see, for example, \cite{Fr}) and the implication ($\ref{cccCorson}) \Rightarrow (\ref{notMAomega1}$) follows from the aforementioned result of Shapirovskii, so let us assume that $MA_{\omega_1}$ fails and deduce the existence of a ccc Corson compactum without a dense metrizable subspace. By \cite{TV}, if $MA_{\omega_1}$ fails there is a ccc poset $\mathbb{P}$ and an uncountable set $E \subset \mathbb{P}$ without an uncountable \emph{centered} subset. Recall that a subset $C$ of $\mathbb{P}$ is called centered if for every finite $F \subset C$ there is $q \in \mathbb{P}$ such that $q \leq p$, for every $p \in F$.

Let 

$$X=\{A \subset E: A \textrm{ is centered in } \mathbb{P} \}$$

 \noindent viewed as a subspace of $2^E$ via identification of each set with its characteristic function. Since centered subsets of $E$ are countable $X$ is in fact a subset of the $\Sigma$-product $\Sigma(2^E)$, thus $X$ is a Corson compactum. Since the compact space $X$ contains the uncountable discrete set $\{\{p\} : p \in \mathbb{P} \}$, it is not metrizable. To see why $X$ is ccc, suppose by contradiction it has an uncountable cellular family $\mathcal{U}$. Without loss of generality we can assume that each element of $\mathcal{U}$ is a basic open set, that is $\mathcal{U}=\{B_\sigma: \sigma \in U\}$, where $U$ is an uncountable family of partial functions from $E$ to $2$ and $B_\sigma=\{x \in X: \sigma \subset x \}$. By a standard $\Delta$-system argument, we can assume that $\sigma \cup \tau$ is a function, for every $\sigma, \tau \in U$. This is equivalent to saying that $B_\sigma \cap B_\tau=\emptyset$ if and only if no centered subset of $E$ contains both $\sigma^{-1}(1)$ and $\tau^{-1}(1)$. For every $\sigma \in U$, let $p_\sigma$ be a common extension of the elements in $\sigma^{-1}(1)$. Then the fact that $\{B_\sigma: \sigma \in U \}$ is a cellular family in $X$ implies that $\{p_\sigma: \sigma \in U\}$ is an uncountable antichain in $\mathbb{P}$ and that is a contradiction.

\end{proof}

We note that the above theorem can also be derived from Theorem 3.4 of \cite{TChain}, which states that $MA_{\omega_1}$ is equivalent to the statement that every compact first-countable ccc space is separable. Indeed, if $X$ is a first countable compactum, then it is a compact space of countable tightness and hence, by a result of Shapirovskii, it admits an irreducible map onto a subspace $Y$ of a $\Sigma$-product $\Sigma(\mathbb{R}^\kappa)$, for some cardinal $\kappa$. Since continuous maps preserve the countable chain condition and irreducible maps preserve density, it follows that $Y$ is a ccc Corson compactum such that $d(Y)=d(X)$. Since every separable Corson compactum is metrizable it turns out that $Y$ is metrizable if and only if $X$ is separable. 

Recall that a poset $(\mathbb{P}, \leq)$ is said to be \emph{powerfully ccc} ((see \cite{TMart})) if each of its finite powers is ccc. Powerfully ccc posets have turned out to be sufficient for all known applications of Martin's Axiom and indeed, it is still unknown whether Martin's Axiom is equivalent to Martin's Axiom restricted to powerfully ccc posets. To guarantee that $X^\omega$ does not have a dense metrizable subspace we need that $X^\omega$ is a non-metrizable ccc Corson compactum.

That leads us to the following variation of the above theorem, which is proved in a similar way:

\begin{theorem} \label{omegaccc}
The following are equivalent:

\begin{enumerate}
\item \label{BcccCorson} There is a Corson compactum $X$ such that $X^\omega$ is ccc but does not have a dense metrizable subspace.
\item \label{BnwdCompact} There is a compact space $X$ such that $X^\omega$ is ccc and can be covered by $\omega_1$ nowhere dense subsets.
\item \label{BnotMAomega1} $MA_{\omega_1}$ for powerfully ccc posets fails.
\item \label{piweightCorson} There is a ccc Corson compactum $X$ of weight $\omega_1$ such that $X^\omega$ does not have a dense metrizable subspace.
\end{enumerate}
\end{theorem}

\begin{proof}
The equivalence of ($\ref{BcccCorson}$), ($\ref{BnwdCompact}$), ($\ref{BnotMAomega1}$) is proved by a similar argument as the one proving Theorem $\ref{cccequiv}$. To prove the equivalence of ($\ref{BnotMAomega1}$) and ($\ref{piweightCorson}$) let us assume that there is a ccc Corson compactum $X$ such that $w(X)=\aleph_1$ and $X^\omega$ does not have a dense metrizable subspace. Then by Theorem $\ref{mainthm}$ $X^\omega$ is ccc, which implies ($\ref{BcccCorson}$). Vice versa, if $MA_{\omega_1}$ for powerfully ccc posets fails, then there is a powerfully ccc poset $\mathbb{P}$ containing an uncountable set $E$ without uncountable centered subsets. Without loss of generality we may assume that $|E|=\aleph_1$. Define:

$$X=\{A \subset E: A \textrm{ is centered in } \mathbb{P} \}$$

Then, by the proof of Theorem $\ref{cccequiv}$, $X \subset \Sigma(2^{\omega_1})$ (and hence $X$ is a Corson compactum of $\pi$-weight $\omega_1$), $X^\omega$ is ccc and $X^\omega$ has no dense metrizable subspace.
\end{proof}

\begin{corollary}
Assume a Suslin tree exists. Then there is a Corson compactum $X$ such that $X^\omega$ does not have a dense metrizable subspace.
\end{corollary}

\begin{proof}
If there is a Suslin tree then $MA_{\omega_1}$ for powerfully ccc posets fails. This follows from the well-known fact that the poset of all finite antichains of a given tree with no uncountable branches with reverse containment is powerfully ccc (see, for example, Lemma 9.2 of \cite{THand}).
\end{proof}

\begin{question}
Is there a natural weakening of Martin's Axiom which is equivalent to the statement that the countable power of every ccc Corson compactum has a dense metrizable subspace?
\end{question}

The most natural class of Corson compacta such that $X^\omega$ is ccc is the class of Corson compacta which support a \emph{strictly positive} measure, that is, a measure $\mu$ such that $\mu(U)>0$, for every non-empty open subset $U$ of $X$. Kunen and van Mill \cite{KvM} proved that the existence of a non-metrizable Corson compactum supporting a strictly positive measure is equivalent to the statement that $MA_{\omega_1}$ for measure algebras fails, a statement considerable stronger than the statement ($\ref{BnotMAomega1}$) of Theorem $\ref{omegaccc}$.

\begin{theorem} \label{KvMTheorem}
(Kunen and van Mill, \cite{KvM}) The following are equivalent:

\begin{enumerate}
\item \label{CorsonMeas}  Every Corson compactum with a strictly positive measure is metrizable.
\item \label{Precaliber}  $MA_{\omega_1}$ for measure algebras.
\end{enumerate}
\end{theorem}

We will provide a proof of the above theorem which is substantially different from Kunen and van Mill's original argument, by using Kelley's criterion for the existence of strictly positive measures in compacta. To state Kelley's criterion we need to define the notion of \emph{intersection number} of a family of sets.

Let $\mathcal{F}$ be a family of subsets of $X$. Let $\vec{F}=(F_1, \dots F_n)$ be a finite sequence of not necessarily distinct elements of $\mathcal{F}$. Define $i(\vec{F})$ to be:

$$i(\vec{F})=\max \{ |J|: J \subset \{1, \dots n\}, \bigcap_{j \in J} F_j \neq \emptyset \}$$

\noindent Then define the \emph{intersection number $\mathcal{I}(\mathcal{F})$ of $\mathcal{F}$} to be:

$$I(\mathcal{F})=\inf \left \{\frac{i(\vec{F})}{|\vec{F}|}: \vec{F} \in \mathcal{F}^n, n < \omega \right \}$$

\noindent where $|\vec{F}|$ denotes the length of $\vec{F}$.

It is easy to see that if $\mu$ is a finitely additive probability measure on $X$ and $\mathcal{F}$ is a family of measurable subsets of $X$ such that for some $\epsilon> 0$, $\mu(F) \geq \epsilon$, for every $F \in \mathcal{F}$, then $I(\mathcal{F}) \geq \epsilon$. In \cite{Ke}, Kelley proved a sort of converse to that, providing a well-known criterion for the existence of a strictly positive measure on a compact space.

\begin{theorem} \label{KelleyTheorem}
(Kelley, \cite{Ke}) A compact space $X$ supports a strictly positive measure if and only if the family of all non-empty open subsets of $X$ can be split into countably many subfamilies with positive intersection number.
\end{theorem}

Recall that a poset $\mathbb{P}$ has \emph{precaliber $\omega_1$} if every uncountable $X \subset \mathbb{P}$ contains an uncountable $Y \subset X$ which is centered in $\mathbb{P}$. The statement that every ccc poset has precaliber $\omega_1$ is a standard consequence of $MA_{\omega_1}$ and is in fact equivalent to it. Similarly, $MA_{\omega_1}$ for measure algebras is equivalent to the statement that every measure algebra has precaliber $\omega_1$. This fact will be key in our proof of Kunen and van Mill's theorem.

\begin{proof}[Proof of Theorem $\ref{KvMTheorem}$]

To prove that $(\ref{Precaliber})$ implies $(\ref{CorsonMeas})$, pick a Corson compactum $X \subset \Sigma(\mathbb{R}^I)$ supporting a strictly positive measure $\mu$. If $X$ is not metrizable then there is $\epsilon > 0$ such that the set

$$J=\{i \in I: (\exists x \in X)(|x(i)| \geq \epsilon) \}$$

\noindent is uncountable. For $j \in J$, let $K_j=\{x \in X: |x(j)| \geq \epsilon \}$. Then $\{K_j: j \in J \}$ is an uncountable family of elements of the measure algebra of $(X, \mu)$ without an uncountable centered subfamily.

Vice versa, assume that $\omega_1$ is not a precaliber of a measure algebra. It is well-known (see Lemma 6 of \cite{TSep}) that this implies that $2^{\omega_1}$ is covered by an increasing sequence $\{N_\alpha: \alpha < \omega_1 \}$ of Haar null sets. By induction on $\alpha < \omega_1$ we pick compact subsets $K_\alpha$ of $[0,1]$ such that:
    
    \begin{enumerate}
    \item $K_\alpha \cap N_\xi=\emptyset$, for all $\xi \leq \alpha$.
    \item $K_\alpha \cap \bigcap_{\xi \in F} K_\xi=\emptyset$, for every finite $F \subset \alpha$ such that $\mu(\bigcap_{\xi \in F} K_\xi)=0$.
    \item $\mu (K_\alpha) > 0$.
    \end{enumerate}

 Let $X$ be the collection of all $A \subset \omega_1$ such that $\{K_\alpha: \alpha \in A \}$ has the finite intersection property. Note that, since $\bigcup \{N_\alpha: \alpha < \omega_1\}$ is an increasing sequence of sets which covers $[0,1]$ and $[0,1]$ is compact, every such $A$ must be countable.
 
 Identifying $X$ with a subset of $2^{\omega_1}$ we observe that $X$ is closed and therefore it is a compact subspace of $2^{\omega_1}$. From the countability of every $A$ it follows that $X$ is actually a subspace of $\Sigma(2^{\omega_1})$ and hence it's a Corson compactum.

The standard neighbourhoods of $X$ have the form $B_p=\{A \in X: (\forall \alpha \in dom(p))(\alpha \in A \iff p(\alpha)=1)\}$ where $p$ is a finite partial function from $\omega_1$ to $2$. Let $\mathcal{C}_{\omega_1}$ be the set of all $p \in Fn(\omega_1,2)$ such that $B_p \neq \emptyset$. According to Theorem $\ref{KelleyTheorem}$, in order to show that $X$ supports a strictly positive measure it suffices to decompose $\{B_p: p \in \mathcal{C}_{\omega_1}\}$ into countably many families having positive intersection number. Since $\mathcal{C}_{\omega_1}$ is a $\sigma$-centered poset it suffices to show that, if $\mathcal{C} \subset \mathcal{C}_{\omega_1}$ is a centered subset then $\{B_p: p \in \mathcal{C}\}$ can be split into countably many subfamilies with positive intersection number.

For every $p \in \mathcal{C}$, let:

$$K_p=\bigcap \{ K_\xi: \xi \in p^{-1}(1)\}$$

Then $\mu(K_p)>0$, for all $p \in \mathcal{C}$ and, given $\mathcal{C}_0 \subset \mathcal{C}$, the intersection number of $\{B_p: p \in \mathcal{C}_0\}$ is equal to the intersection number of $\{K_p: p \in \mathcal{C}_0\}$. Thus, if for $n<\omega$, we let:

$$\mathcal{C}_n=\left \{p \in \mathcal{C}: \mu(K_p) \geq \frac{1}{n+1}\right \}$$

\noindent then the intersection number of $\{B_p: p \in \mathcal{C}_n \}$ is equal to the intersection number of $\{K_p: p \in \mathcal{C}_n \}$ and this is at least $\frac{1}{n+1}$.

\end{proof}






\section{Corson compacta and singular cardinals}

A family $\mathcal{F}$ of countable sets is called \emph{locally countable} (see \cite{T}) if for every uncountable subfamily $\mathcal{G} \subset \mathcal{F}$, the set $\bigcup \mathcal{G}$ is uncountable. The existence of a locally countable family of countable subsets of $\aleph_\omega$ which has cardinality $>\aleph_\omega$ is independent of ZFC. It follows from $\largesquare_{\aleph_\omega}$ or, more generally, from the existence of a good PCF scale at $\aleph_\omega$, but is negated by Chang's Conjecture for $\aleph_\omega$ (see, for example, \cite{KMS} and \cite{T}).

\begin{theorem} \label{TodEx} 
Suppose that for some singular cardinal $\kappa$ of countable cofinality, the set $[\kappa]^\omega$ contains a locally countable family of cardinality $>\kappa$. Then there is a Corson compactum $X$ such that $c(X^\omega)=\kappa < d(X)$ (and hence $X^\omega$ has no dense metrizable subspace).
\end{theorem}

\begin{proof}
Let $\mathcal{F}$ be a locally countable family of countable subsets of $\kappa$ such that $|\mathcal{F}| > \kappa$. Without loss of generality we can assume that each element of $\mathcal{F}$ has order type $\omega$, so we can identify $a \in \mathcal{F}$ with its increasing enumeration $a: \omega \to \kappa$ and $\mathcal{F}$ with the corresponding subset of $\kappa^\omega$. Given $a, b \in \mathcal{F}$ we denote $\Delta(a,b)$ to be the least integer $n<\omega$ such that $a(n) \neq b(n)$. We identify each subset of $\mathcal{F}$ with its characteristic function and define our space $X$ to be the following subspace of $2^{\mathcal{F}}$:

$$X=\{A \subset \mathcal{F}: (\forall \{a,b,c\} \in [A]^3)(|\{\Delta(a,b), \Delta(b,c), \Delta(a,c)\}| > 1 \}$$ 

It's easy to see that $X$ is a closed subset of $2^{\mathcal{F}}$ and hence compact. From the fact that $\mathcal{F}$ is locally countable it follows that each element of $X$ is countable. Indeed, suppose by contradiction that $A$ is an uncountable element of $X$. Then by the properties of $\mathcal{F}$, $\bigcup A$ is uncountable, so there must be $n<\omega$ such that $\{a(n): a \in A \}$ is uncountable. Let $n$ be the least integer with respect to this property. Note that $\{a \upharpoonright n, a \in A\}$ is a countable set. By the pigeonhole principle there must be a set $B \subset A$ such that $\{a(n): a \in B\}$ is uncountable and $|\{a \upharpoonright n: a \in B\}|=1$. Picking $a,b,c \in B$ we see that $|\{\Delta(a,b), \Delta(b,c), \Delta(a,c)\}|=1$. It turns out that $A \notin X$.

Therefore $X$ is a Corson compactum.

To show that $c(X^\omega) \leq \kappa$, we prove the stronger fact that the standard base of $X$ is $\kappa$-linked. Given a partial function $s \in Fn(\mathcal{F},2)$ and $n<\omega$ we denote by $s \upupharpoons n$ the set $\{(a \upharpoonright n, s(a)): a \in dom(s)\}$ and $k_s=\max \{\Delta(a,b): a, b \in dom(s) \}$. Let $s,t \in Fn(\mathcal{F}, 2)$, and $n<\omega$ be an integer such that $n> \min \{k_s, k_t\}$; it is easy to see that if $s \upupharpoons n = t \upupharpoons n$ then $[s] \cap [t] \neq \emptyset$.

Therefore for every partial function $\sigma \in Fn(Fn(\kappa, \omega), \omega)$, the set $\mathcal{U}_{\sigma}=\{[s]: s \upupharpoons (k_s+1) = \sigma \}$ is linked. To finish note that the standard basis $\mathcal{U}$ of $X$ satisfies $\mathcal{U} = \bigcup \{\mathcal{U}_\sigma: \sigma \in Fn(Fn(\kappa, \omega), \omega)\}$.
\end{proof}

\begin{question}
Is there a ZFC example of a Corson compact space $X$ such that $c(X^\omega)<d(X)$?
\end{question}

If we are willing to relax a little bit the requirement that $X$ is a Corson compactum, we can obtain an example in ZFC. Recall that a family of sets $\mathcal{F}$ is called $(\mu, \kappa)$-\emph{sparse} (see \cite{KMS}) if for every subfamily $\mathcal{G} \subset \mathcal{F}$ such that $| \mathcal{G} | \geq \mu$ we have $|\bigcup \mathcal{G}| \geq \kappa$. So a locally countable family is just an $(\omega_1, \omega_1)$-sparse family. 

The following lemma is a consequence of Shelah's result that every PCF scale at $\aleph_\omega$ contains a club set of good points of cofinality $\geq \aleph_4$ (a direct proof of this lemma is given in \cite{KMS}).

\begin{lemma}
In ZFC there is an $(\aleph_4, \aleph_1)$-sparse family of cardinality $>\aleph_\omega$ in $[\aleph_\omega]^\omega$.
\end{lemma}

A space $X$ is called a $\tau$-\emph{Corson compactum} (see, for example, \cite{BM}) if $X$ is a compact subspace of $\Sigma_\tau(\mathbb{R}^\kappa)=\{x \in \mathbb{R}^\kappa: |supp(x)| \leq \tau\}$. Thus an $\aleph_0$-Corson compactum is simply a Corson compactum. The proof of the following theorem is essentially the same as the proof of Theorem $\ref{TodEx}$, using an $(\aleph_4, \aleph_1)$-sparse family in $[\aleph_\omega]^\omega$ instead of a locally countable family.

\begin{theorem}
In ZFC there is an $\aleph_3$-Corson compactum $X$ such that $c(X^\omega)<d(X)$.
\end{theorem}

\section{Concluding remarks}

In this final section we show that the usual examples of Corson compacta without a dense metrizable subspace cannot serve as counterexamples to Question $\ref{mainquest}$ and make various remarks about the existence of dense metrizable subspaces in their finite powers.

Recall that a space is $d$-separable if it has a $\sigma$-discrete dense subspace. Of course if a space has a dense metrizable subspace then it is $d$-separable, but the converse is not true, not even for Corson compacta. To get a counterexample, it suffices to consider the adequate version of Todorcevic's Corson compactum from \cite{Le} (Example 1).

\begin{theorem}
Let $X$ be Todorcevic's Corson compactum. Then:

\begin{enumerate}
\item \label{finpow} $X^n$ is not $d$-separable, for every $n<\omega$.
\item \label{countpow} $X^\omega$ has a dense metrizable subspace.
\end{enumerate}
\end{theorem}

\begin{proof}
Item $\ref{finpow}$ was proved in \cite{SSS}. Item $\ref{countpow}$ follows from Theorem $\ref{mainthm}$ since $\hat{c}(X) > w(X)=2^{\aleph_0}$.
\end{proof}

Recall that a Suslin subtree $T$ of $\omega^{<\omega_1}$ is called \emph{coherent} if for every $s, t \in T$, $\{\beta \in dom{(s)} \cap dom{(t)}: s(\beta) \neq t(\beta) \}$ is finite. Let $\mathbb{K}_T$ be the Suslin continuum corresponding to a coherent Suslin tree $T$. In \cite{THand} it is shown that, while $\mathbb{K}_T$ has no dense metrizable subspace, its square $\mathbb{K}_T^2$ has a dense metrizable subspace. Using Shapirovskii's Theorem we can then map $\mathbb{K}_T$ irreducibly onto a compact subspace $X$ of some $\Sigma$-product of lines. Thus we have the following:

\begin{theorem}
If there is a coherent Suslin tree then there is a perfectly normal Corson compactum $X$ such that $X$ has no dense metrizable subspace but $X^2$ does.
\end{theorem}

To achieve a higher control of the existence of dense metrizable subspaces in finite powers of Suslin continua, we need the notion of a \emph{full Suslin tree} from \cite{THand}. 

Given a tree $(T, \leq_T)$, for every $t \in T$, we denote by $T^t$ the set $\{s \in T: t \leq_T s \}$. 

Recall that if $\{T_i: i=1, \dots n \}$ is a sequence of trees then $\bigotimes_{i=1}^n T_i$ denotes the set of all sequences $(t_1, \dots t_n)$ such that $t_i \in T_i$ for every $i \in \{1, \dots n\}$ and $t_1, \dots t_n$ have the same height. Equip $\bigotimes_{i=1}^n T_i$ with the following partial order: $(t_1, \dots t_n) \leq (t'_1, \dots t'_n)$ if and only if $t_i \leq_{T_i} t'_i$, for every $i \in \{1, \dots n\}$. Then $\bigotimes_{i=1}^n T_i$ is a tree.

For a positive integer $n$ we say that a Suslin tree $T$ is \emph{$n$-full} if $\bigotimes_{i=1}^n T^{t_i}$ is Suslin, for every sequence $t_1, \dots t_n$ of distinct elements of $T$ having the same height. We say that a Suslin tree $T$ is \emph{full} if it is $n$-full, for every $1 \leq n < \omega$. Needless to say, coherent and full Suslin trees exist consistently. For example, they were constructed first by Jensen using $\Diamond$ (see Section 6 of \cite{THand}). 

\begin{theorem}
If there is a full Suslin tree than there is a perfectly normal Corson compactum $X$ such that:

\begin{enumerate}
\item \label{nodense} $X^n$ has no dense metrizable subspace, for every $n<\omega$.
\item \label{yesdense} $X^\omega$ has a dense metrizable subspace.
\end{enumerate}
\end{theorem}

\begin{proof}
Let $T$ be a full Suslin tree and $\mathbb{K}_T$ be the corresponding Suslin continuum. By Shapirovskii's Theorem there is an irreducible map $f$ from $\mathbb{K}_T$ onto a Corson compactum $X$. Item $\ref{nodense}$ is proved with the same argument proving Lemma 6.7 of \cite{THand}. As for item $\ref{yesdense}$, note that $\hat{c}(\mathbb{K}_T^2)> \omega_1=\pi w(\mathbb{K}_T^2)$ and therefore, since irreducible map preserve cellularity and $\pi$-weight $\hat{c}(X^2)> \omega_1= \pi w(X^2)=w(X)$, therefore the result follows from Theorem $\ref{mainthm}$.
\end{proof}

For a positive integer $n$, we say that a tree $T$ is \emph{$n$-coherent} if $T$ can be represented as a downwards-closed subtree of $\omega^{<\omega_1}$ such that, for all $t_1, \dots t_{n+1}$ having the same height $\alpha$, the set:

$$\{\xi < \alpha: t_i(\xi) \neq t_j(\xi)\}$$

\noindent is finite, for all distinct $i, j \in \{1, \dots, n+1 \}$. Thus $1$-coherent is the same as coherent.  The constructions of full and coherent Suslin trees presented in section 6 of \cite{THand} also give us an $n$-full, $n$-coherent Suslin tree, for every $1 \leq n < \omega$. Therefore we have the following theorem.

\begin{theorem}
Let $n$ be a positive integer and let us assume the existence of an $n$-full, $n$-coherent Suslin tree. Then there is a perfectly normal Corson compactum $X$ such that:

\begin{enumerate}
\item $X^n$ has no dense metrizable subspace.
\item $X^{n+1}$ has a dense metrizable subspace.
\end{enumerate}
\end{theorem}

\begin{question}
Is there a ZFC example of a Corson compactum $X$ such that $X^n$ has no dense metrizable subspace but $X^{n+1}$ has a dense metrizable subspace?
\end{question}

\begin{question}
Is there an example of a Corson compactum $X$ such that no odd power of $X$ has a dense metrizable subspace while every even power of $X$ has a dense metrizable subspace?
\end{question}

\section*{Acknowledgements}
The second author is grateful to INdAM-GNSAGA for partial financial support and to the Center for Advanced Studies in Mathematics at Ben Gurion University for partial financial support during his visit to Be'er Sheva in 2018. The research of the third author is partially supported by grants from NSERC (455916) and CNRS (UMR7586).

\end{document}